\documentclass[a4paper, reqno, 12pt]{amsart}

\usepackage[usenames,dvipsnames]{color}
\usepackage{amsthm,amsfonts,amssymb,amsmath,amsxtra}
\usepackage[all]{xy}
\SelectTips{cm}{}
\usepackage{xr-hyper}
\usepackage[colorlinks=
   citecolor=Black,
   linkcolor=Red,
   urlcolor=Blue]{hyperref}
\usepackage{verbatim}

\usepackage[margin=1.25in]{geometry}
\usepackage{mathrsfs}

\RequirePackage{xspace}
\RequirePackage{etoolbox}
\RequirePackage{varwidth}
\RequirePackage{enumitem}
\RequirePackage{tensor}
\RequirePackage{mathtools}
\RequirePackage{longtable}
\RequirePackage{multirow}

\def\ge{\geqslant}
\def\le{\leqslant}
\def\a{\alpha}

\def\g{\gamma}
\def\G{\Gamma}
\def\d{\delta}
\def\D{\Delta}

\def\o{\omega}

\def\s{\sigma}
\def\t{\tau}

\def\k{\kappa}
\def\l{\lambda}

\def\i{^{-1}}
\def\co{\mathcal O}

\def\<{\langle}
\def\>{\rangle}

\newcommand{{\BG}}{\ensuremath{\mathbb {G}}\xspace}

\newcommand{{\BK}}{\ensuremath{\mathbb {K}}\xspace}

\newcommand{\BN}{\ensuremath{\mathbb {N}}\xspace}

\newcommand{\BQ}{\ensuremath{\mathbb {Q}}\xspace}
\newcommand{\BR}{\ensuremath{\mathbb {R}}\xspace}
\newcommand{\BS}{\ensuremath{\mathbb {S}}\xspace}

\newcommand{\BZ}{\ensuremath{\mathbb {Z}}\xspace}

\DeclareMathOperator{\Gal}{Gal}

\DeclareMathOperator{\rank}{rank}

\def\tW{\tilde W}
\def\tS{\tilde \BS}

\DeclareMathOperator{\supp}{supp}

\newtheorem{theorem}{Theorem}
\newtheorem{proposition}[theorem]{Proposition}

\newtheorem{corollary}[theorem]{Corollary}

\theoremstyle{definition}

\newtheorem{remark}[theorem]{Remark}

\numberwithin{equation}{section}
\numberwithin{theorem}{section}

\setitemize[0]{leftmargin=*,itemsep=\the\smallskipamount}
\setenumerate[0]{leftmargin=*,itemsep=\the\smallskipamount}

\renewcommand{\to}{%
   \ifbool{@display}{\longrightarrow}{\rightarrow}%
   }
\let\shortmapsto\mapsto
\renewcommand{\mapsto}{%
   \ifbool{@display}{\longmapsto}{\shortmapsto}%
   }
\newlength{\olen}
\newlength{\ulen}
\newlength{\xlen}
\newcommand{\xra}[2][]{%
   \ifbool{@display}%
      {\settowidth{\olen}{$\overset{#2}{\longrightarrow}$}%
       \settowidth{\ulen}{$\underset{#1}{\longrightarrow}$}%
       \settowidth{\xlen}{$\xrightarrow[#1]{#2}$}%
       \ifdimgreater{\olen}{\xlen}%
          {\underset{#1}{\overset{#2}{\longrightarrow}}}%
          {\ifdimgreater{\ulen}{\xlen}%
             {\underset{#1}{\overset{#2}{\longrightarrow}}}
             {\xrightarrow[#1]{#2}}}}%
      {\xrightarrow[#1]{#2}}
   }
\makeatother
\newcommand{\xyra}[2][]{%
   \settowidth{\xlen}{$\xrightarrow[#1]{#2}$}%
   \ifbool{@display}%
      {\settowidth{\olen}{$\overset{#2}{\longrightarrow}$}%
       \settowidth{\ulen}{$\underset{#1}{\longrightarrow}$}%
       \ifdimgreater{\olen}{\xlen}%
          {\mathrel{\xymatrix@M=.12ex@C=3.2ex{\ar[r]^-{#2}_-{#1} &}}}%
          {\ifdimgreater{\ulen}{\xlen}%
             {\mathrel{\xymatrix@M=.12ex@C=3.2ex{\ar[r]^-{#2}_-{#1} &}}}
             {\mathrel{\xymatrix@M=.12ex@C=\the\xlen{\ar[r]^-{#2}_-{#1} &}}}}}%
      {\mathrel{\xymatrix@M=.12ex@C=\the\xlen{\ar[r]^-{#2}_-{#1} &}}}%
   }
\makeatletter
\newcommand{\xla}[2][]{%
   \ifbool{@display}%
      {\settowidth{\olen}{$\overset{#2}{\longleftarrow}$}%
       \settowidth{\ulen}{$\underset{#1}{\longleftarrow}$}%
       \settowidth{\xlen}{$\xleftarrow[#1]{#2}$}%
       \ifdimgreater{\olen}{\xlen}%
          {\underset{#1}{\overset{#2}{\longleftarrow}}}%
          {\ifdimgreater{\ulen}{\xlen}%
             {\underset{#1}{\overset{#2}{\longleftarrow}}}
             {\xleftarrow[#1]{#2}}}}%
      {\xleftarrow[#1]{#2}}
   }
\newcommand{\isoarrow}{%
   \ifbool{@display}{\overset{\sim}{\longrightarrow}}{\xrightarrow\sim}%
   }
   
\begin{document}

\title[]{Cordial elements and dimensions of affine Deligne-Lusztig varieties}
\author[X. He]{Xuhua He}
\address[]{The Institute of Mathematical Sciences and Department of Mathematics, The Chinese University of Hong Kong, Shatin, N.T., Hong Kong.}
\email{xuhuahe@gmail.com}
\thanks{}

\keywords{Affine Deligne-Lusztig varieties}
\subjclass[2010]{14L05, 20G25}

\date{\today}

\begin{abstract}
The affine Deligne-Lusztig variety $X_w(b)$ in the affine flag variety of a reductive group $\mathbb G$ depends on two parameters: the $\s$-conjugacy class $[b]$ and the element $w$ in the Iwahori-Weyl group $\tW$ of $\mathbb G$. In this paper, for any given $\s$-conjugacy class $[b]$, we determine the nonemptiness pattern and the dimension formula of $X_w(b)$ for most $w \in \tW$. 
\end{abstract}

\maketitle

\tableofcontents

\section{Introduction}
\subsection{Motivation} The notion of affine Deligne-Lusztig variety was first introduced by Rapoport in \cite{Ra-guide}. It plays an important role in arithmetic geometry and the Langlands program. One of the main motivations comes from the reduction of Shimura varieties. We focus on the affine Deligne-Lusztig varieties in the affine flag variety in this paper. In this case, the affine Deligne-Lusztig varieties are closely related to the Shimura varieties with Iwahori level structure. On the special fibers, there are two important stratifications: 

\begin{itemize}

\item Newton stratification, indexed by specific $\s$-conjugacy classes $[b]$ in the associated $p$-adic group. 

\item Kottwitz-Rapoport stratification, indexed by specific elements $w$ in the associated Iwahori-Weyl group. 
\end{itemize}

A fundamental question is to determine when the intersection of a Newton stratum indexed by $[b]$ and a Kottwitz-Rapoport stratum indexed by $w$ is nonempty and to determine its dimension. Such intersection is closely related to the affine Deligne-Lusztig variety $X_w(b)$, see for instance \cite{HR}. In a parallel story over function fields, affine Deligne-Lusztig varieties also arise naturally in the study of local shtukas, see for instance \cite{HV}. 

Motivated by the study of Shimura varieties and the local shtukas, one would like to understand the following fundamental questions on the affine Deligne-Lusztig varieties: 

\begin{itemize}
\item When is the affine Deligne-Lusztig variety nonempty?

\item If nonempty, what is its dimension?
\end{itemize}

It is also worth pointing out that much information on the affine Deligne-Lusztig varieties in the partial affine flag varieties (which are closely related to Shimura varieties with other parahoric level structures) can be deduced from the information on the affine Deligne-Lusztig varieties in the affine flag variety. 

\subsection{The main result} In this paper we determine, for any given $\s$-conjugacy class $[b]$, the nonemptiness pattern and dimension formula of $X_w(b)$ for most $w$ in the Iwahori-Weyl group $\tW$. To state the result, we introduce some notations first. For simplicity, we will only consider the split groups $\BG$ here. The general case will be studied in the main context. 

The Iwahori-Weyl group $\tW$ is the semidirect product of the coweight lattice with the relative Weyl group $W_0$. We may write $\tW$ as $\tW=\sqcup_{\l \text { is dominant}} W_0 t^\l W_0$. For any $w \in W_0 t^\l W_0$, we set $\l_w=\l$. The $\s$-conjugacy classes $[b]$ are classified by Kottwitz \cite{Ko1} and \cite{Ko2} via the two invariants: the image under the Kottwitz map $\k$ and the Newton point $\nu_b$ (which is a dominant rational coweight). By the Mazur's inequality for the affine Deligne-Lusztig varieties in the affine Grassmannian \cite{Ga}, we deduce that if $X_w(b) \neq \emptyset$, then $\k(w)=\k(b)$ and $\l_w \ge \nu_b$ with respect to the dominance order of the rational coweights. 

The converse, however, is far from being true. The main result of this paper is the following. 

\begin{theorem}
Let $w \in \tW$. Suppose that $w$ is in a Shrunken Weyl chamber. If $\k(w)=\k(b)$, $\l_w-\nu_b$ is a linear combination of the simple coroots with all the coefficients positive, and $\l_w^{\flat\mkern-2.4mu\flat} \ge \nu_b$, then we have a complete description of the nonemptiness pattern and dimension formula for $X_w(b)$. 
\end{theorem}

We refer to \S\ref{ADLV} for the definition of Shrunken Weyl chambers, \S\ref{double-flat1} for the definition of  $-^{\flat\mkern-2.4mu\flat}$ and Theorem \ref{main} for the precise description of the nonemptiness pattern and dimension formula. These assumptions are satisfied for example when $\l_w \ge \nu_b+2 \rho^\vee$, where $\rho^\vee$ is the half sum of positive coroots. See Corollary \ref{7.4}. 

\subsection{Some previous results}
In \cite[Conjecture 9.5.1]{GHKR2}, G\"ortz, Haines, Kottwitz and Reuman made several influential conjectures on the nonemptiness pattern and dimension formula of $X_w(b)$.  

First, for the basic $\s$-conjugacy class $[b]$, they gave a conjecture in \cite[Conjecture 9.5.1 (a)]{GHKR2} on the nonemptiness pattern and dimension formula for $X_w(b)$ for $w$ in the Shrunken Weyl chamber. This conjecture was established in \cite{He14}. For $X_w(b)$ with $[b]$ basic, and $w$ outside the Shrunken Weyl chamber, in \cite[Conjecture 9.4.2]{GHKR2} they gave a conjecture on the nonemptiness pattern. This conjecture is established in \cite{GHN}. But for $[b]$ basic and $w$ outside the Shrunken Weyl chamber, no conjectural dimension formula of $X_w(b)$ has even been formulated so far. 

For arbitrary $\s$-conjugacy class $[b]$, they made an interesting conjecture in \cite[Conjecture 9.5.1 (b)]{GHKR2} which predicts the difference of the dimensions of $X_w(b)$ and $X_w(b_{basic})$, where $[b_{basic}]$ is the unique basic $\s$-conjugacy class such that $\k(b)=\k(b_{basic})$. In this conjecture, $w$ is not required to be shrunken, but the length of $w$ is required to be big enough with some (unspecified) lower bound. In the later works, we studied $X_w(b)$ via a somehow different direction. First, the assumption that $w$ is in the Shrunken Weyl chamber is added, as even for the basic $\s$-conjugacy classes, the dimension formula for the affine Deligne-Lusztig varieties with $w$ outside the shrunken Weyl chamber is still very mysterious. Second, one would like to have a specific lower bound on $w$. 

For split groups and the case where $[b]$ is represented by translation elements, under the ``very shrunken'' assumption, the nonemptiness pattern and the dimension formula of $X_w(b)$ were given in \cite[Theorem 2.28 \& Theorem 2.34]{He-CDM}. A similar result were obtained in \cite{MST} under a different condition on $w$. 

For other nonbasic $\s$-conjugacy classes, little is known so far on the nonemptiness pattern and dimension formula of $X_w(b)$. 

\subsection{Old Strategies}
We discuss several strategies used to study the nonemptiness pattern and dimension formula for $X_w(b)$ in the previous works. 

The emptiness pattern is established via the method of $P$-alcove elements introduced in \cite{GHKR2}. The upper bound of $\dim X_w(b)$ is given by the virtual dimension $d_w(b)$ introduced in \cite{He14}. 

In \cite{He14}, we combined the Deligne-Lusztig reduction with some remarkable properties of minimal length elements in their conjugacy classes in $\tW$ to establish a method to compute $\dim X_w(b)$ for arbitrary $w$ and arbitrary $[b]$. As a consequence, we establish the ``dimension=degree'' theorem which relates the dimension of affine Deligne-Lusztig varieties with the degree of the class polynomials of the affine Hecke algebras. However, the computation of the class polynomials, in general, is extremely difficult. The ``dimension=degree'' theorem does not lead to explicit descriptions of the nonemptiness pattern and the dimension formula of $X_w(b)$. 

For basic $[b]$, assume that $X_w(b) \neq \emptyset$. It remains to show that $\dim X_w(b)$ reaches the upper bound $d_w(b)$. Note that for any Coxeter element $c$, $\dim X_c(b)$ is easy to compute. This will be used as the starting point. In \cite{He14}, we constructed an explicit ``reduction path'' from an element $w$ in the shrunken Weyl chamber to an element $w'$ with finite part a Coxeter element. By \cite{HY}, the minimal length elements in the conjugacy class of $w'$ in $\tW$ are the Coxeter elements $c$. This gives a ``reduction path'' from $w$ to $c$ and thus leads to a lower bound of $\dim X_w(b)$. Fortunately, the lower bound also equals to the virtual dimension $d_w(b)$. Thus we proved the nonemptiness pattern and the dimension formula of $X_w(b)$ with basic $[b]$. 

For split groups and the case where $[b]$ is represented by translation elements, in \cite{He-CDM} we used the superset method of \cite{GHKR2} to relate the nonemptiness pattern and dimension formula of $X_w(b)$ with $X_{w'}(1)$ for a given $w'$. Note that $[1]$ is a basic $\s$-conjugacy class. We then used the result on $X_{w'}(1)$ established in \cite{He14} to obtain the desired result on $X_w(b)$. A very different approach is introduced in \cite{MST}, where the authors used the alcove walk and Littelman path to study the nonemptiness pattern and dimension formula of $X_w(b)$.

I do not know how/whether the methods in \cite{He-CDM} or in \cite{MST} for the translation elements may be generalized to arbitrary $\s$-conjugacy class $[b]$. The reduction method introduced in \cite{He14}, in theory, works for an arbitrary $\s$-conjugacy class $[b]$. However, to construct an explicit ``reduction path'' from a given $w$ to a minimal length element associated to a nonbasic $[b]$ is very challenging. Q. Yu has written down a computer program to construct the ``reduction path'' for groups with small ranks. But so far it is not clear how such a ``reduction path'' may be constructed in general. 

\subsection{New strategy}
The new strategy in this paper is as follows. Instead of using minimal length elements as the starting point, we use the cordial elements introduced by Mili\'cevi\'c and Viehmann in \cite{MV} as the starting point instead. In section \ref{3}, we construct a new family of cordial elements. For any element $w'$ in this family, $\dim X_{w'}(b)$ equals to the virtual dimension. We then construct in section \ref{4} an explicit ``reduction path'' from an element $w$ in the shrunken Weyl chamber to an element in this family. This is where the assumption $\l_w^{\flat\mkern-2.4mu\flat} \ge \nu_b$ is used. This shows that $\dim X_w(b) \ge d_w(b)$. Finally we use the result that $\dim X_w(b) \le d_w(b)$ established in \cite{He14} and \cite{He-CDM} to prove the desired nonemptiness pattern and the dimension formula of $X_w(b)$.  

\subsection{Acknowledgment} We thank S. Nie, E. Viehmann and Q. Yu for helpful discussions. The paper was written down during the visit to H. Bao at NUS. We thank the excellent working environment it provides. We also thank U. G\"ortz, S. Nie and E. Viehmann for useful comments on a previous version of the paper. 
 
\section{Preliminary}

\subsection{The reductive group $\BG$ and its Iwahori-Weyl group}
Let $F$ be a non-archimedean local field and $\breve F$ be the completion of the maximal unramified extension of $F$. We write $\Gamma$ for $\Gal(\overline F/F)$, and write $\Gamma_0$ for the inertia subgroup of $\Gamma$.

Let $\mathbb G$ be a connected reductive group over $F$. Let $\s$ be the Frobenius morphism of $\breve F/F$. We write $\breve G$ for $\mathbb G(\breve F)$. We use the same symbol $\s$ for the induced Frobenius morphism on $\breve G$. 

We fix a maximal $\breve F$-split torus $S$ in $\mathbb G$ defined over $F$ which contains a maximal $F$-split torus. Let $T$ be the centralizer of $S$ in $\mathbb G$. Then $T$ is a maximal torus. Let $\mathcal A$ be the apartment of $\mathbb G_{\breve F}$ corresponding to $S_{\breve F}$. Thus $\mathcal A$ is (non-canonically) isomorphic to $V=X_*(T)_{\Gamma_0} \otimes_{\mathbb Z} \mathbb R$. The Frobenius $\s$ naturally acts on $\mathcal A$. We fix a $\s$-stable alcove $\mathfrak a$ in $\mathcal A$, and let $\breve I \subset \breve G$ be the Iwahori subgroup corresponding to $\mathfrak a$. Thus $\breve I$ is $\s$-stable. 

We denote by $N$ the normalizer of $T$ in $\mathbb G$. The \emph{relative Weyl group} $W_0$ is defined to be $N(\breve F)/T(\breve F)$. The \emph{Iwahori--Weyl group} (associated to $S$) is defined as 
$$\tW= N(\breve F)/T(\breve F) \cap \breve I.$$ For any $w \in \tW$, we choose a representative $\dot w$ in $N(L)$. 

We have a natural short exact sequence $0 \to X_*(T)_{\Gamma_0} \to \tW \to W_0 \to 0$. We choose a special vertex of $\mathfrak a$ and represent $\tW$ as a semidirect product $$\tW=X_*(T)_{\G_0} \rtimes W_0=\{t^\l w; \l \in X_*(T)_{\G_0}, w \in W_0\}.$$

The Iwahori-Weyl group $\tW$ contains the affine Weyl group $W_a$ as a normal subgroup and we have $$\tW=W_a \rtimes \Omega,$$ where $\Omega$ is the normalizer of $\mathfrak a$. The length function $\ell$ and Bruhat order $\le$ on $W_a$ extend in a natural way to $\tW$. The Frobenius $\s$ naturally acts on $\tW$, in such a way that the subset $\tS \subset \tW$ is stable. 

For any $K \subset \tS$, we denote by $W_K$ the subgroup of $\tW$ generated by $s \in K$. Let ${}^K \tW$ (resp. $\tW^K$) be the set of minimal length elements in their cosets in $W_K \backslash \tW$ (resp. $\tW/W_K$). 

Let $\BS \subset \tS$ be the set of simple reflections of $W_0$. Since $\s$ preserves $W_0$, it also preserves $\BS$. By convention, the dominant Weyl chamber of $V$ is opposite to the unique Weyl chamber containing $\mathfrak a$. Let $\D$ be the set of relative simple roots determined by the dominant Weyl chamber. Then $\s(\D)=\D$. For any $s \in \BS$, we denote by $\a_s \in \D$ the corresponding simple root and $\a_s^\vee$ the corresponding simple coroot. We denote by $w_\BS$ the longest element of $W_0$. 

We define the $\s$-conjugation action on $\breve G$ by $g \cdot_\s g'=g g' \s(g) \i$. Let $B(\mathbb G)$ be the set of $\s$-conjugacy classes on $\breve G$. The classification of the $\s$-conjugacy classes is obtained by Kottwitz in \cite{Ko1} and \cite{Ko2}. Any $\s$-conjugacy class $[b]$ is determined by two invariants: 
\begin{itemize}
	\item The element $\k([b]) \in \Omega_{\s}$; 
	
	\item The Newton point $\nu_b \in \big((X_*(T)_{\Gamma_0, \BQ})^+\big)^{\langle\sigma\rangle}$. 
\end{itemize}

Here $-_\s$ denotes the $\s$-coinvariants, 
$(X_*(T)_{\Gamma_0, \BQ})^+$ denotes the intersection of  $X_*(T)_{\Gamma_0}\otimes \BQ=X_*(T)^{\Gamma_0}\otimes \BQ$ with the set $X_*(T)_\BQ^+$ of dominant elements in $X_*(T)_\BQ$; the action of $\sigma$ on $(X_*(T)_{\Gamma_0, \BQ})/W_0$ is transferred to an action on $(X_*(T)_\BQ)^+$ ({\it L-action}). 

For any $w \in \tW$, we write $\k(w)$ for $\k(\dot w)$. It is easy to see that $\k(w)$ is independent of the choice of the representative $w$. 

\subsection{Affine Deligne-Lusztig varieties}\label{ADLV} We have the following generalization of the Bruhat decomposition $$\breve G=\sqcup_{w \in \tW} \breve I \dot w \breve I,$$ due to Iwahori and Matsumoto \cite{IM} in the split case, and to Bruhat and Tits \cite{BT} in the general case. Let $Fl=\breve G/\breve I$ be the {\it affine flag variety}. For any $b \in \breve G$ and $w \in \tW$, we define the corresponding {\it affine Deligne-Lusztig variety} in the affine flag variety $$X_w(b)=\{g I \in \breve G/\breve I; g \i b \s(g) \in \breve I \dot w \breve I\} \subset Fl.$$

As discussed in \cite[\S 2]{GHN}, the study of nonemptiness pattern and dimension formula of affine Deligne-Lusztig varieties for arbitrary reductive group may be reduced to simple and quasi-split groups over $F$. From now on, we assume that $\mathbb G$ is simple and quasi-split over $F$. In this case, the $\s$-action on $\tW$ preserves $W_0$ and $X_*(T)_{\G_0}$ respectively. 

Now we recall the definition of virtual dimension in \cite[\S10.1]{He14}. 

Note that any element $w \in \tW$ may be written in a unique way as $w=x t^\mu y$ with $\mu$ dominant, $x, y \in W_0$ such that $t^\mu y \in {}^{\BS} \tW$. In this case, we set $$\eta_\s(w) = \s^{-1}(y) x.$$

Let $\mathbb J_b$ be the reductive group over $F$ with $\mathbb J_b(F)=\{g \in \breve G; g b \s(g) \i=b\}.$ The \emph{defect} of $b$ is defined by $\text{def}(b)=\rank_F \mathbb G-\rank_F \mathbb J_b.$ Here for a reductive group $\mathbb H$ defined over $F$, $\rank_F$ is the $F$-rank of the group $\mathbb H$. Let $\rho$ be the dominant weight with $\<\a^\vee, \rho\>=1$ for any $\a \in \D$. The \emph{virtual dimension} is defined to be $$d_w(b)=\frac 12 \big( \ell(w) + \ell(\eta_\s(w)) -\text{def}(b)  \big)-\<\nu_b, \rho\>.$$

The following result is proved in \cite[Corollary 10.4]{He14} for residually split groups and proved in \cite[Theorem 2.30]{He-CDM} for the general case. 

\begin{theorem}\label{dim-vir}
Let $b \in \breve G$ and $w \in \tW$. Then $\dim X_w(b) \le d_w(b)$. 
\end{theorem}

For any $w \in W_0$, we denote by $\supp(w) \subset \BS$ the set of simple reflections appears in some (or equivalently, any) reduced expression of $w$. We set $\supp_\s(w)=\cup_{i \in \BZ} \s^i(\supp(w))$. 

For any $w \in \tW$, let $\l_w$ is the unique dominant coweight such that $w \in W_0 t^{\l_w} W_0$. For any $\l \in X_*(T)_{\G_0}$, we denote by $\l^\diamondsuit$ the average of the $\s$-orbit of $\l$. For any $\l, \l' \in X_*(T)^+_{\BQ}$, we write $\l \ge \l'$ if $\l-\l' \in \sum_{\a \in \D} \BQ_{\ge 0} \a^\vee$ and write $\l \ge _\BZ \l'$ if $\l-\l' \in \sum_{\a \in \D} \BN \a^\vee$. Here $\BN$ is the set of natural numbers, i.e., the set of nonnegative integers. 

A {\it critical strip} of the apartment $V$ by definition is the subset $\{v; -1<\<v, \a\><0\}$ for a given positive root in the reduced root system associated to the affine Weyl group $W_a$. We remove all the critical strips from $V$ and call each connected component of the remaining subset of $V$ {\it a Shrunken Weyl chamber}.

\section{Some combinatorial properties}

\subsection{Minimal length elements} For any $\s$-conjugacy class $\co$ in $\tW$, we denote by $\co_{\min}$ the set of minimal length elements in $\co$. For $w, w' \in \tW$ and $s \in \tS$, we write $w \xrightarrow{s}_\s w'$ if $w'=s w \s(s)$ and $\ell(w') \le \ell(w)$. We write $w \to_\s w'$ if there is a sequence $w=w_0, w_1, \cdots, w_n=w'$ of elements in $\tW$ such that for any $k$, $w_{k-1} \xrightarrow{s}_\s w_k$ for some $s \in \tS$. We write $w \approx_\s w'$ if $w \to_\s w'$ and $w' \to_\s w$. It is easy to see that $w \approx_\s w'$ if $w \to_\s w'$ and $\ell(w)=\ell(w')$.

The following result is proved in \cite[\S 2]{HN14}. 

\begin{theorem}\label{ux}
	Let $\co$ be a $\s$-conjugacy class of $\tW$ and $w \in \co$. Then there exists $w' \in \co_{\min}$ such that $w \to_\s w'$.
\end{theorem}

\begin{theorem}\label{2.2}
	Let $b \in \breve G$ and $w \in \co_{\min}$ for some $\s$-conjugacy class $\co$ of $\tW$. Then $X_w(b) \neq \emptyset$ if and only if $\dot w \in [b]$. In this case, $\dim X_w(b)=\ell(w)-\<\nu_b, 2 \rho\>$. 
\end{theorem}

\subsection{Deligne-Lusztig reduction}
Now we recall the ``reduction'' \`a la Deligne and Lusztig for affine Deligne-Lusztig varieties (see \cite[Proof of Theorem 1.6]{DL}, also \cite[Corollary 2.5.3]{GH}).

\begin{proposition}\label{DLReduction}
Let $b \in \breve G$. Then 
\begin{enumerate}
\item 
Let $w, w' \in \tW$ with $w \approx_\s w'$, then $$\dim X_w(b)=\dim X_{w'}(b).$$ 
\item
Let $w \in \tW$ and $s \in \tS$ with $\ell(s w \s(s))=\ell(w)-2$. Then $$\dim X_w(b)=\max\{\dim X_{s w}(b), \dim X_{s w \s(s)}(b)\}+1.$$ \end{enumerate}
\end{proposition}

Here by convention, we set $\dim \emptyset=-\infty$ and $-\infty+n=-\infty$ for any $n \in \BR$. 

\subsection{The relation $\Rightarrow$} Following \cite[Definition 3.1.4]{GH}, for $w, w' \in \tW$, we write $w \Rightarrow_\s w'$ if for any $b \in \breve G$, $$\dim X_w(b)-d_w(b) \ge \dim X_{w'}(b)-d_{w'}(b).$$ Here again by convention, we set $\dim \emptyset=-\infty$. If the right hand side is $-\infty$, then the inequality holds regardless of the left hand side. It is also easy to see that the relation is transitive. 

Note that by definition of virtual dimension, $w \Rightarrow_\s w'$ if and only if for any $b \in \breve G$ with $X_{w'}(b) \neq \emptyset$, $X_w(b) \neq \emptyset$, and in this case, $$\dim X_w(b)-\dim X_{w'}(b) \ge \frac{1}{2}\bigl( \ell(w) + \ell(\eta_\s(w)) -\ell(w')-\ell(\eta_\s(w'))\bigr).$$

We write $w \Leftrightarrow_\s w'$ if $w \Rightarrow_\s w'$ and $w' \Rightarrow_\s w$.

\subsection{The monoid structure on $\tW$}
By \cite{He0}, for any $w, w' \in \tW$,  the subset $\{u w'; u \le w\}$ of $\tW$ contains a unique maximal element which we denote by $w \ast w'$. Moreover, $w \ast w'=\max\{u v; u \le w, v \le w'\}$. Hence $\ast$ is associative. This gives a monoid structure on $\tW$. If $w_1 \le w$ and $w'_1 \le w'$, then $w_1 \ast w'_1 \le w \ast w'$. 

\section{The cordial elements}\label{3}

\subsection{Definition}
There is a natural partial ordering $\le$ on $B(\mathbb G)$ defined as follows. Let $[b], [b'] \in B(\mathbb G)$. Then $[b] \le [b']$ if $\k(b)=\k(b')$ and $\nu_b \le \nu_{b'}$. 

Now we recall the cordial elements introduced by Mili\'cevi\'c and Viehmann in \cite{MV}. 

For any $w \in \tW$, there is a unique maximal $\s$-conjugacy class $[b]$ such that $X_w(b) \neq \emptyset$. We denote this $\s$-conjugacy class by $[b_w]$. The element $w$ is called \emph{cordial} if $\dim X_w(b_w)=d_w(b_w)$. Equivalently, $w$ is cordial if and only if $\ell(w)-\ell(\eta_\s(w))=\langle\nu_{b_w}, 2 \rho\rangle-\text{def}(b_w)$. 

By definition, if $w \Leftrightarrow_\s w'$, then $w$ is a cordial element if and only if $w'$ is a cordial element. The following result is proved in \cite[Theorem 1.1 \& Corollary 2.18]{MV}\footnote{In fact, in \cite{MV} only the unramified groups are considered. However, as shown in \cite[Section 6]{He14} the nonemptiness pattern and dimension formula of affine Deligne-Lusztig varieties for any quasi-split group are the same as those for a suitable unramified group.}. 

\begin{theorem}\label{cordial}
Let $w \in \tW$ be a cordial element. Then 

(1) Let $[b], [b'] \in B(\mathbb G)$. If $[b] \le [b'] \le [b_w]$ and $X_w(b) \neq \emptyset$, then $X_w(b') \neq \emptyset$. 

(2) If $X_w(b) \neq \emptyset$, then $\dim X_w(b)=d_w(b)$. 
\end{theorem}

It is mentioned in loc. cit that to fully characterize the cordial elements is fairly difficult. In \cite[Theorem1.2]{MV}, some interesting families of cordial elements are provided. The main result of this section is to provide another family of cordial elements. 

\begin{theorem}\label{cordial-1}
Let $\l$ be a dominant coweight and $x \in W_0$. Then $x t^\l$ is a cordial element and $[b_{x t^\l}]=[\dot t^\l]$. 
\end{theorem}

\begin{remark}
	The original proof I had is a bit technical. The following proof is suggested to me by E. Viehmann. 
\end{remark}

\subsection{Mazur's inequality} Recall that $\BG$ is quasi-split over $F$. Let $\breve K \supset \breve I$ be a $\s$-stable  special maximal parahoric subgroup of $\breve G$. The nonemptiness pattern of the affine Deligne-Lusztig varieties in the affine Grassmannian $\breve G/\breve K$ is determined in terms of Mazur's inequality. It was established by Gashi \cite{Ga} for unramified groups and proved in general case in \cite[Theorem 7.1]{He14}. We may reformulate the result as follows. 

\begin{theorem}\label{mazur}
	Let $\l$ be a dominant coweight and $b \in \breve G$. Then $[b] \cap \breve K \dot t^\l \breve K \neq \emptyset$ if and only if $\k(b)=\k(t^\l)$ and $\nu_b \le \l^\diamondsuit$. 
\end{theorem}

\subsection{Proof of Theorem \ref{cordial-1}} 

Let $w \in \tW$. By definition, $[b_w]$ is the unique maximal $\s$-conjugacy class that intersects $\breve I \dot w \breve I$. By \cite[Corollary 5.6]{Vie-Ann}, $[b_w]$ is also the unique maximal $\s$-conjugacy class that intersects $\overline{\breve I \dot w \breve I}$. Note that $$\breve I \dot w_\BS \dot t^\l \breve I \subset \breve K \dot t^\l \breve K \subset \overline{\breve I \dot w_\BS \dot t^\l \breve I}.$$ By Theorem \ref{mazur}, $[b_{w_\BS t^\l}]=[\dot t^\l]$.

On the other hand, $t^\l$ is a minimal length element in its $\s$-conjugacy class. Hence by Theorem \ref{2.2}, $[b_{t^\l}]=[\dot t^\l]$. 

Since $t^\l \le x t^\l \le w_{\BS} t^\l$ for any $x \in W_0$, we have that $$\breve I \dot t^\l \breve I \subset \overline{\breve I \dot x \dot t^\l \breve I} \subset \overline{\breve I \dot w_\BS \dot t^\l \breve I}.$$

Hence $[\dot t^\l]=[b_{t^\l}] \le [b_{x t^\l}] \le [b_{w_{\BS} t^\l}]=[\dot t^\l]$. Thus $[b_{x t^\l}]=[\dot t^\l]$. 

Now $\nu_{b_{x t^\l}}=\l^\diamondsuit$ and $\text{def}(b_{x t^\l})=0$. Hence $$\ell(x t^\l)-\ell(\eta_\s(x t^\l))=\ell(x)+\ell(t^\l)-\ell(x)=\ell(t^\l)=\langle \l, 2\rho\rangle=\langle \l^\diamondsuit, 2 \rho\rangle.$$ Thus $x t^\l$ is a cordial element. 

\subsection{Another family of cordial element} 
Let $w \in \tW$ such that $w \mathfrak a$ is in the antidominant Weyl chamber, i.e. $w=w_{\BS} t^\l y$, where $\l$ is a dominant coweight and $y \in W_0$ with $t^\l y \in {}^{\BS} \tW$. Then $\eta_\s(w)=\s \i(y) w_{\BS}$. Set $w'=\s \i(y) w_{\BS} t^\l$. Then $\eta_\s(w')=\eta_\s(w)$. Moreover, it is easy to see that $w \approx_\s w'$. Hence $w \Leftrightarrow_\s w'$. By Theorem \ref{cordial-1}, $w'$ is a cordial element. Hence 

(a) $w$ is also a cordial element. 

This is first proved by Mili\'cevi\'c and Viehmann in \cite[Theorem 1.2 (a)]{MV}\footnote{In \cite{MV}, only the split group is considered. However, the proof of \cite[Theorem 1.2 (a)]{MV} works for any quasi-split groups as well.}. 

It is also worth mentioning that not all the elements of the form $x t^\l$ is $\approx_\s$-equivalent to an element in the antidominant Weyl chamber. 

\section{From $w$ to a cordial element}\label{4}

We first show that 

\begin{proposition}
	Let $\l, \l'$ be dominant coweights. Then the set $$\{\mu'; \mu' \text{ is dominant}, \mu'+\l' \ge_\BZ \l\}$$ contains a unique minimal element with respect the dominance order $\ge_\BZ$. 
\end{proposition}

\begin{remark}
	The proof is due to S. Nie. 
\end{remark}

\begin{proof}
	Let $\mu'_1, \mu'_2$ be dominant coweights with $\mu'_1+\l' \ge_\BZ \l$ and $\mu'_2+\l' \ge_\BZ \l$. We may write $\mu'_1-\mu'_2$ as $\mu'_1-\mu'_2=\g_1-\g_2$, where $\g_1 \in \sum_{\a \in J_1} \BZ_{>0} \a$, $\g_2 \in \sum_{\a \in J_2} \BZ_{>0} \a$ for some $J_1, J_2 \subset \D$ with $J_1 \cap J_2=\emptyset$. 
	
	Set $\mu=\mu'_1-\g_1=\mu'_2-\g_2$. Let $\a \in \D$. Since $J_1 \cap J_2=\emptyset$, we have $\a \notin J_1$ or $\a \notin J_2$. If $\a \notin J_1$, then $\<\mu, \a\> \ge \<\mu'_1, \a\> \ge 0$. If $\a \notin J_2$, then $\<\mu, \a\> \ge \<\mu'_2, \a\> \ge 0$. Thus $\mu$ is dominant. By definition, $\mu'_1 \ge_\BZ \mu$ and $\mu'_2 \ge_\BZ \mu$. Moreover, $$\l'-\l+\mu'_1=\l'-\l+\mu'_2+\g_1-\g_2 \in (\sum_{\a \in \D} \BZ_{>0} \a+\g_1-\g_2) \cap \sum_{\a \in \D} \BZ_{>0} \a.$$ Since $J_1 \cap J_2=\emptyset$, we have $\l'-\l+\mu'_1-\g_1 \in \sum_{\a \in \D} \BZ_{>0} \a$. In other words, $\l'+\mu \ge_\BZ \l$. 
	
	The statement is proved. 
\end{proof}

\subsection{The normalized subtraction}
For any dominant coweights $\l, \l'$, we denote by $\l-_{\text{dom}} \l'$ the unique minimal element in the set $$\{\mu'; \mu' \text{ is dominant}, \mu'+\l' \ge_\BZ \l\}.$$ It is easy to see that if $\l-\l'$ is dominant, then $\l-_{\text{dom}} \l'=\l-\l'$. We call $-_{\text{dom}}$ the {\it normalized subtraction}. Now we prove some properties of the normalized subtraction $-_{\text{dom}}$. 

\begin{corollary}\label{cor-1}
	Let $\l, \l'$ be dominant coweights. Let $\l''$ be a dominant coweight with $\l' \ge_\BZ \l''$. Let $x \in W_0$ and let $\mu$ be the unique dominant coweight in the $W_0$-orbit of $\l-x(\l'')$. Then $\mu \ge_\BZ \l-_{\text{dom}} \l'$. 
\end{corollary}

\begin{proof}
	Note that $\mu-(\l-x(\l'')) \in \sum_{\a \in \D} \BN \a$, $\l''-x(\l'') \in \sum_{\a \in \D} \BN \a$ and $\l'-\l'' \in \sum_{\a \in \D} \BN \a$. Thus \begin{align*} \mu+\l' &=(\mu-\l+x(\l''))+\l-x(\l'')+\l' \\ &=(\mu-\l+x(\l''))+(\l''-x(\l''))+(\l'-\l'')+\l \\ & \ge_\BZ \l.\qedhere\end{align*}
\end{proof}

\begin{corollary}\label{cor-2}
	Let $\l, \l_1, \l_2$ be dominant coweights. Then $$(\l-_{\text{dom}} \l_1)-_{\text{dom}} \l_2=\l-_{\text{dom}} (\l_1+\l_2).$$
\end{corollary}

\begin{proof}
	Set $\mu_1=(\l-_{\text{dom}} \l_1)-_{\text{dom}} \l_2$ and $\mu_2=\l-_{\text{dom}} (\l_1+\l_2)$. By definition, $$(\l_1+\l_2)+\mu_1=\l_1+(\l_2+\mu_1) \ge_\BZ \l_1+(\l-_{\text{dom}} \l_1) \ge_\BZ \l.$$ So $\mu_1 \ge_\BZ \mu_2.$
	
	On the other hand, $$\l_1+(\l_2+\mu_2)=(\l_1+\l_2)+\mu_2 \ge_\BZ \l.$$ So by definition, $\l_2+\mu_2 \ge_\BZ \l-_{\text{dom}} \l_1$ and $\mu_2 \ge_\BZ \mu_1$. 
\end{proof}

\subsection{The double flat operator}\label{double-flat1} For any subset $J$ of $\BS$, we denote by $\rho^\vee_J$ the dominant coweight with $$\<\rho^\vee_J, \a_s\>=\begin{cases} 1, & \text{ if } s \in J; \\ 0, & \text{ if } s \notin J.\end{cases}$$ Let $\eta^\vee_J$ be the unique dominant coweight in the $W_0$-orbit of $-\s \i(\rho^\vee_J)$. 

Let $w \in \tW$. We write $w$ as $w=x t^\l y$ with $\l$ dominant, $x, y \in W_0$ and $t^\l y \in {}^{\BS} \tW$. Let $J=\{s \in \BS; s y<y\}$. Since $t^{\l} y \in {}^{\BS} \tW$, we have $\<\l, \a_s\>>0$ for any $s \in J$. In particular, $\l-\rho^\vee_J$ is dominant.  We set $$\l_w^{\flat\mkern-2.4mu\flat}=(\l-\rho^\vee_J)-_{\text{dom}} \eta^\vee_J=\l-_{\text{dom}} (\rho^\vee_J+\eta^\vee_J).$$

The main result of this section is as follows. 

\begin{theorem}\label{rightarrow}
Assume that $\mathbb G$ is simple and quasi-split over $F$. Let $w \in \tW$. Suppose that $\supp_\s(\eta_\s(w))=\BS$. Then there exists a dominant coweight $\g$ with $\g \ge_\BZ \l_w^{\flat\mkern-2.4mu\flat}$ and $a \in W_0$ with $\supp_\s(a)=\BS$ such that $$w \Rightarrow_\s a t^\g.$$ 
\end{theorem}

\subsection{A convenient notation} Following \cite[\S 2.4]{GH}, we give a convenient notation for varieties of tuples of elements in $Fl$. We explain the notation by examples. Let $\co_w=\{(g \breve I, g \dot w \breve I); g \in \breve G\} \subset Fl\times Fl$. Then we set
\begin{align*}
\{ \xymatrix{ g \ar[r]^-w & g'' \ar[r]^-{w'} & g' } \}=\{  (g, g', g'') \in (Fl)^3;\ (g,g'')\in \co_{w},\ (g'',g')\in \co_{w'} \}.
\end{align*}
Similarly,
\begin{align*}
\{ \xymatrix{ g \ar[r]^-w \ar@/_2pc/[rr]_-{w''} & g'' \ar[r]^-{w'} & g' } \}=\{ ( g, g', g'') \in (Fl)^3;\ (g,g'')\in \co_{w},\ (g'',g')\in \co_{w'},\ (g,g')\in\co_{w''} \}.
\end{align*}
The affine Deligne-Lusztig varieties can be written as
\[
X_w(b) = \{ \xymatrix{ g \ar[r]^-w & b\s(g) } \}.
\]
In all these cases, we do not distinguish between the sets given by the conditions on the relative position, and the corresponding locally closed sub-ind-schemes of the product of affine flag varieties. The following result is proved in \cite[Proposition 2.5.2]{GH}. 

\begin{proposition}\label{GHDL}
Let $w, w'\in \tW$, and let $w'' \in \{ ww', w \ast w' \}$. Then the map 
\[
\pi\colon\{ \xymatrix{ g \ar[r]^-w \ar@/_2pc/[rr]_-{w''} & g'' \ar[r]^-{w'} & g' } \}
\longrightarrow
\{ \xymatrix{ g \ar[r]^-{w''} & g'}\}, \qquad (g, g', g'') \mapsto (g, g') 
\] is surjective. Moreover all the fibers have dimension
\[
\dim\pi^{-1}( ( g, g') ) \ge \left\{
\begin{array}{ll}
\ell(w) + \ell(w') - \ell(w*w'), & \text{if } w'' = w \ast w';\\
\frac 12  \bigl( \ell(w) + \ell(w') - \ell(ww') \bigr), & \text{if } w'' = ww'.
\end{array}
\right.
\]
\end{proposition}

\subsection{Proof of Theorem \ref{rightarrow}} We write $w$ as $w=x t^{\l_w} y$ with $x, y \in W_0$ and $t^{\l_w} y \in {}^{\BS} \tW$. Let $J=\{s \in \BS; s y<y\}$. 

Let $J'=\{s \in \BS; s(\l_w-\rho^\vee_J)=\l_w-\rho^\vee_J\}$. We write $\s \i(y) x$ as $\s \i(y)x=x' z$ for some $w \in W_0^{J'}$ and $z \in W_{J'}$. Let $\g$ be the unique dominant coweight in the $W_0$-orbit of $\l_w-\rho^\vee_J+(x') \i \s \i(\rho^\vee_J)$.  By Corollary \ref{cor-1}, $\g \ge_\BZ \l_w^{\flat\mkern-2.4mu\flat}$.

Let $K=\{s \in \BS; s(\g)=\g\}$ and $y' \in W_0^K$ with $\l_w-\rho^\vee_J+(x') \i \s \i(\rho^\vee_J) =y'(\g)$. 

Set $w_1=x z \i t^{\l_w-\rho^\vee_J} y'$ and $w_2=(y') \i z t^{\rho^\vee_J} y$. Then $w=w_1 w_2$. Set $a =((y')\i z) * \s(x' y')$. Then \begin{align*} \supp_\s(a) &=\supp_\s((y') \i z) \cup \supp_\s(\s(x'y'))=\supp_\s((y') \i z) \cup \supp_\s(x'y') \\ & \supset \supp_\s(x' y' (y') \i z)=\supp_\s(x' z)=\supp_\s(\eta_\s(w))=\BS.\end{align*}

By the proof of \cite[Proposition 11.5]{He14}, $\ell(w)=\ell(w_1)+\ell(w_2)$, $\ell((y') \i z)+\ell(x' y')=\ell(\s \i(y) x)$ and $\supp_\s(a) = \BS$. 

Now
\[
X_{w}(b)= \{ \xymatrix{ g \ar[r]^-w & b \sigma(g) } \}=\{  \xymatrix{ g \ar[r]^-{w_1} & g_1 \ar[r]^-{w_2} & b \sigma(g) }\}.
\]
Set
\begin{align*}
X_1= &  \{ \xymatrix{ g_1 \ar[r]^-{w_2} & g_2 \ar[r]^-{\s(w_1)} & b \sigma(g_1) } \} \\
\cong & \{ \xymatrix{ g_1 \ar[r]^-{(y')\i z} &  g_3 \ar[r]^-{t^{\rho^\vee_J} y} & g_2 \ar[r]^-{\s(w_1)} & b \sigma(g_1)  } \}.
\end{align*}
The map $(g, g_1) \mapsto (g_1, b \sigma(g))$ is a universal homeomorphism from $X_{x}(b)$ to $X_1$.
We have that $y \s(x z \i)=\s(\s \i(y) x z \i)=\s(x')$ and $$t^{\rho^\vee_J} y \s(w_1)=t^{\rho^\vee_J} \s(x' t^{\l_w-\rho^\vee_J} y')=\s(x') \s(t^{y'(\g)}) \s(y')=\s(x' y' t^\g).$$ 
Let 
\begin{gather*}
X_2=\{ \xymatrix{  g_1 \ar[r]^-{(y')\i z} &  g_3 \ar[r]^-{t^{\rho^\vee_J} y} \ar@/_2pc/[rr]_-{\s(w y' t^\g)} & g_2 \ar[r]^-{\s(w_1)} & b \sigma(g_1) } \} \subset X_1; \\
X_3=\{ \xymatrix{ g_1  \ar[r]^-{(y')\i z} & g_3  \ar[r]^-{\s(x' y' t^\g)} & b \s(g_1)}  \}=\{ \xymatrix{ g_1 \ar[r]^-{(y')\i z} & g_3 \ar[r]^-{\s(x' y')} &  g_4 \ar[r]^-{\s(t^\g)} & b \s(g_1)}\},
\end{gather*}

By Proposition \ref{GHDL},
$$\dim(X_{x}(b)) \ge \dim(X_2) \ge \dim(X_3)+\frac{\ell(t^{\rho^\vee_J} y)+\ell(w_1)-\ell(x' y' t^\g)}{2}.$$ 

Recall that $a=(y \i z)*(x' y')$. We set
\begin{gather*}
X_4=\{ \xymatrix{ g_1 \ar[r]^-{(y')\i z} \ar@/_2pc/[rr]_-{a} & g_3 \ar[r]^{\s(x' y')} &  g_4 \ar[r]^-{\s(t^\g)} & b \s(g_1)} \} \subset X_3, \\
X_5=\{ \xymatrix{ g_1 \ar[r]^-a & g_4 \ar[r]^-{\s(t^\g)} & b \s(g_1)}\}.
\end{gather*}

By Proposition \ref{GHDL}, $$\dim(X_3) \ge \dim(X_5)+\ell((y') \i z)+\ell(x' y')-\ell(a)=\ell(\eta_\s(w))-\ell(a).$$ 

Notice that $\ell(a t^\g)=\ell(a)+\ell(t^\g)$. Thus the map $(g_1, g_4) \mapsto g_1$ gives an isomorphism $X_5 \cong X_{a t^\g}(b)$. If $X_{a t^\g}(b)\ne \emptyset$, then $X_{x}(b)\ne\emptyset$. Note that $\ell((y') \i z)+\ell(x' y')=\ell(\s \i(y) x)=\ell(\eta_\s(w))$. Therefore \begin{align*} & \dim X_w(b)-\dim X_{a t^\g}(b) \\ & \ge \frac{\ell(t^{\rho^\vee_J} y)+\ell(w_1)-\ell(x' y' t^\g)}{2}+\ell(\eta_\s(w))-\ell(a)\\ &=\frac{\ell(t^{\rho^\vee_J} y)+\ell((y') \i z)+\ell(w_1)-\ell(x' y' t^\g)+\ell(x' y')-\ell(a)}{2}+\frac{\ell(\eta_\s(w))}{2}-\frac{\ell(a)}{2} \\ & =\frac{\ell(w_2)+\ell(w_1)-\ell(a t^\g)}{2}+\frac{\ell(\eta_\s(w))}{2}-\frac{\ell(a)}{2} \\ &=\frac{\ell(w)-\ell(a t^\g)}{2}+\frac{\ell(\eta_\s(w))}{2}-\frac{\ell(a)}{2}=d_{w}(b)-d_{a t^\g}(b).\end{align*}

So $w \Rightarrow_\s a t^\g$. The theorem is proved.

\section{Proof of main theorem}

Now we state our main result. 

\begin{theorem}\label{main}
	Suppose that $\BG$ is simple and quasi-split over $F$. Let $b \in \breve G$ and $w \in \tW$ such that $w \mathfrak a$ is in a Shrunken Weyl chamber, $\l_w^\diamondsuit-\nu_b \in \sum_{\a \in \D} \BQ_{>0} \a^\vee$ and $(\l_w^{\flat\mkern-2.4mu\flat})^\diamondsuit \ge \nu_b$. Then $X_w(b) \neq \emptyset$ if and only if $\k(b)=\k(w)$ and $\supp_\s(\eta_\s(w))=\BS$. In this case, $\dim X_w(b)=d_w(b)$. 
\end{theorem}

\begin{remark}
It is worth mentioning that in most cases, $\l_w-\l_w^{\flat\mkern-2.4mu\flat}$ is dominant and nonzero. In this case, $\l_w^\diamondsuit-(\l_w^{\flat\mkern-2.4mu\flat})^\diamondsuit \in \sum_{\a \in \D} \BQ_{>0} \a^\vee$. However, if $\BG$ is split over $F$ and $\l_w$ is a minuscule coweight, then $\l_w^{\flat\mkern-2.4mu\flat}=\l_w$. Thus the assumption $\l_w^\diamondsuit-\nu_b \in \sum_{\a \in \D} \BQ_{>0} \a^\vee$ is needed in our statement. 
\end{remark}

We first prove the theorem and then discuss the assumptions in the statement. In particular, we will give a simple condition where the assumptions are satisfied in Corollary \ref{7.4}. 

\subsection{The $(J, w, \d)$-alcove elements}
We recall the alcove elements introduced in \cite{GHKR2} for split groups and then generalized to quasi-split groups in \cite{GHN}. 

For any $J \subset \BS$ with $\s(J)=J$, we denote by $\mathbb M_J \subset \mathbb G$ the standard Levi subgroup corresponding to $J$ and $\mathbb P_J \supset \mathbb M_J$ be the standard parabolic subgroup. Let $\mathbb U_{\mathbb P_J}$ be the unipotent radical of $\mathbb P_J$. 

Let $J \subset \BS$ with $\s(J)=J$ and $x \in W_0$. Let $w\in\tW$.  We say that $w$  is a \emph{$(J, x, \s)$-alcove element} if $x \i w \s(x) \in \tW_J$ and $\mathbb {}^{\dot x} \mathbb U_{\mathbb P_J}(\breve F) \cap {}^{\dot w} \breve I \subseteq \mathbb {}^{\dot x} \mathbb U_{\mathbb P_J}(\breve F) \cap \breve I$. The following result is proved in \cite[Corollary 3.6.1]{GHN}\footnote{In loc.cit., we put the assumption that $[b]$ is basic. In fact, the assumption is required in \cite[Propositon 3.5.1 \& Remark 3.6.2]{GHN} but is not need in \cite[Corollary 3.6.1]{GHN}.}. 

\begin{theorem}\label{alcove-empty}
Let $[b] \in B(\mathbb G)$ and $w \in \tW$. Suppose that $w$ is a $(J, x, \s)$-alcove element. Let $\k_{\mathbb M_J}$ be the Kottwitz map for the group $\mathbb M_J$. If $\k_{\mathbb M_J}(x \i w \s(x)) \neq \k_{\mathbb M_J}(b')$ for any $b' \in [b] \cap \mathbb M_J(\breve F)$, then $X_w(b)=\emptyset$. 
\end{theorem}

\subsection{The emptiness pattern} Suppose that $w \mathfrak a$ is in a Shrunken Weyl chamber and $(\l_w^{\flat\mkern-2.4mu\flat})^\diamondsuit \ge \nu_b$. We write $w$ as $w=x t^\l y$ with $x, y \in W_0$ and $t^\l y \in {}^{\BS} \tW$. If $\k(b) \neq \k(w)$, then $X_w(b)=\emptyset$.  

Now suppose that $\k(b)=\k(w)$ and $\supp_\s(\s \i(y) x)\neq \BS$.  Set $J=\supp_\s(\s \i(y) x)$. By \cite[Lemma 3.6.3]{GHN}, $w$ is a $(J, \s \i(y), \s)$-alcove element. Let $b' \in [b] \cap \mathbb M_J(\breve F)$. We denote by $\nu^{\mathbb M_J}_{b'}$ the image of $b'$ under the Newton map for $\mathbb M_J$. Then $\nu^{\mathbb M_J}_{b'} \in W_0 (\nu_b)$. Hence $\nu_b-\nu^{\mathbb M_J}_{b'} \in \sum_{\a \in \D} \BQ_{\ge 0} \a^\vee$. 

By assumption, $\l^\diamondsuit-\nu_b \in \sum_{\a \in \D} \BQ_{>0} \a^\vee$. Thus $\l^\diamondsuit-\nu^{\mathbb M_J}_{b'} \in \sum_{\a \in \D} \BQ_{>0} \a^\vee$ and can not be written as a linear combination of the coroots in $\mathbb M_J$. Therefore $\k_{\mathbb M_J}(\s \i (y) w y \i) \neq \k_{\mathbb M_J}(b')$. By Theorem \ref{alcove-empty}, $X_w(b)=\emptyset$. 

\subsection{Dimension formula} Suppose that $\k(w)=\k(b)$ and $\supp_\s(\eta_\s(w))=\BS$. By Theorem \ref{rightarrow} (1), there exists a dominant coweight $\g \ge_\BZ \l_w^{\flat\mkern-2.4mu\flat}$ and $a \in W_0$ with $\supp_\s(a)=\BS$ such that $$w \Rightarrow_\s a t^\g.$$ 

By our assumption, $\g^\diamondsuit \ge (\l_w^{\flat\mkern-2.4mu\flat})^\diamondsuit \ge \nu_b$. By Theorem \ref{cordial-1}, $[\dot t^\g]=[b_{a t^\g}]$. Since $\k(w)=\k(t^\g)=\k(b)$, we have $[b] \le [\dot t^\g]$. 

By \cite[Theorem 2.27]{He-CDM}, $X_{a t^\g}(\dot \t) \neq \emptyset$, where $\t \in \Omega$ with $\k(w)=\k(t^\g)=\k(\t)$. Since $\k(w)=\k(b)$, we have $[\dot \t] \le [b]$. 

By Theorem \ref{cordial-1}, $a t^\g$ is a cordial element. Hence by Theorem \ref{cordial} (1), $X_{a t^\g}(b) \neq \emptyset$ and by Theorem \ref{cordial} (2), $\dim X_{a t^\g}(b)=d_{a t^\g}(b)$. 

So by definition of $\Rightarrow_{\s}$, we have $X_w(b) \neq \emptyset$ and $$\dim X_w(b)-d_w(b) \ge \dim X_{a t^\g}(b)-d_{a t^\g}(b)=0.$$ Hence $\dim X_w(b) \ge d_w(b)$. On the other hand, by Theorem \ref{dim-vir}, $\dim X_w(b) \le d_w(b)$. So $\dim X_w(b)=d_w(b)$. 

\subsection{Some remarks on the condition $(\l_w^{\flat\mkern-2.4mu\flat})^\diamondsuit \ge \nu_b$} We first consider the case where $[b]$ is basic. In this case, the condition $(\l_w^{\flat\mkern-2.4mu\flat})^\diamondsuit \ge \nu_b$ follows directly from the condition $\k(b)=\k(w)$. 

Now we consider the nonbasic $[b]$. Suppose that $\l_w^\diamondsuit \ge \nu_b+2 \rho^\vee$. In this case, although $\l_w-2 \rho^\vee$ may not be dominant, its $\s$-average is dominant and is larger than or equal to $\nu_b$. By definition, $\l_w^{\flat\mkern-2.4mu\flat}-(\l_w-\rho^\vee_J-\eta^\vee_J) \in \sum_{\a \in \D} \BQ_{\ge 0} \a^\vee$ for some $J$. Note that $2 \rho^\vee_J-\rho^\vee_J-\eta^\vee_J \in \sum_{\a \in \D} \BQ_{\ge 0} \a^\vee$. We have $\l_w^{\flat\mkern-2.4mu\flat}-(\l_w-2 \rho^\vee) \in \sum_{\a \in \Delta} \BQ_{\ge 0} \a^\vee$. Hence $(\l_w^{\flat\mkern-2.4mu\flat})^\diamondsuit \ge \l_w^\diamondsuit-2 \rho^\vee \ge \nu_b$. It is also easy to see that $\l_w^\diamondsuit-\nu_b \in \sum_{\a \in \D} \BQ_{>0} \a^\vee$. 

In particular, if $\l_w=n \o^\vee$, where $\o^\vee$ is a fundamental coweight and $n \gg 0$ with respect to $[b]$. Then $\l_w^\diamondsuit \ge \nu_b+2 \rho^\vee$ and hence the condition $(\l_w^{\flat\mkern-2.4mu\flat})^\diamondsuit \ge \nu_b$ is satisfied in this case. 

\begin{corollary}\label{7.4}
	Suppose that $\BG$ is simple and quasi-split over $F$. Let $b \in \breve G$ and $w \in \tW$ such that $w \mathfrak a$ is in a Shrunken Weyl chamber. Suppose that  $\l_w^\diamondsuit \ge \nu_b+2 \rho^\vee$. Then $X_w(b) \neq \emptyset$ if and only if $\k(b)=\k(w)$ and $\supp_\s(\eta_\s(w))=\BS$. In this case, $\dim X_w(b)=d_w(b)$. 
\end{corollary}

\subsection{A side remark} By Theorem \ref{mazur}, if $X_w(b) \neq \emptyset$, then $\k(b)=\k(w)$ and $\nu_b \le \l_w^\diamondsuit$. 

Let $w \in \tW$ such that $w \mathfrak a$ is in a Shrunken Weyl chamber. If $\supp_\s(\eta_\s(w))=\BS$, then Theorem \ref{main} describes the nonemptiness pattern and the dimension formula of $X_w(b)$ for most of the $\s$-conjugacy classes $[b]$ with $\k(b)=\k(w)$ and $\nu_b \le \l_w^\diamondsuit$. 

If $\supp_\s(\eta_\s(w))=J \subsetneqq \BS$, then by \cite[Lemma 3.6.3]{GHN}, $w$ is a $(J, \s \i(y), \s)$-alcove element for some $y \in W_0$. Then the Hodge-Newton decomposition (see \cite[Theorem 2.1.4]{GHKR2} for the split group and \cite[Propositon 2.5.1 \& Theorem 3.3.1]{GHN} in general) deduces the study of $X_w(b)$ to the study of a suitable affine Deligne-Lusztig variety associated to the Levi subgroup $\mathbb M_J$. One may apply Theorem \ref{main} to the latter one. In this way, one also obtains an explicit description of the nonemptiness pattern and the dimension formula of $X_w(b)$ for most of the $\s$-conjugacy classes $[b]$ with $\k(b)=\k(w)$ and $\nu_b \le \l_w^\diamondsuit$.

\end{document}